\documentclass{amsart}%
\usepackage{amsmath}
\usepackage{amsfonts}
\usepackage{amssymb}
\usepackage{graphicx}%
\providecommand{\U}[1]{\protect\rule{.1in}{.1in}}
\newtheorem{theorem}{Theorem}
\theoremstyle{plain}

\newtheorem{corollary}{Corollary}

\newtheorem{example}{Example}

\newtheorem{lemma}{Lemma}

\newtheorem{proposition}{Proposition}
\newtheorem{remark}{Remark}

\numberwithin{equation}{section}
\begin{document}
\title{A Note On Unital Full Amalgamated Free Products of Quasidiagonal C*-Algebras }
\author{Qihui Li}
\address{Department of Mathematics, East China University of Science and Technology,
Meilong Road 130, 200237 Shanghai, China.}
\email{lqh991978@gmail.com}
\author{Don Hadwin}
\address{Department of Mathematics, University of New Hampshire, Durham, NH 03824, USA\\
URL: http://euclid.unh.edu/\symbol{126}don}
\email{don@unh.edu}
\author{Jiankui Li}
\address{Department of Mathematics, East China University of Science and Technology,
Meilong Road 130, 200237 Shanghai, China.}
\email{s: jiankuili@yahoo.com}
\author{Xiujuan Ma}
\address{Department of Mathematics, Hebei University of Technology, Tianjing, 300130, China}
\email{mxjsusan@hebut.edu.cn}
\author{Junhao Shen}
\address{Department of Mathematics, University of New Hampshire, Durham, NH 03824, USA}
\email{junhao.shen@unh.edu}
\thanks{The research of the first author is partially supported by National Natural
Science Foundation of China (Grant No. 11201146) and the Fundamental Research
Funds for the Central Universities as well as SRF for ROCS, SEM. }

\begin{abstract}
In the paper, we consider the question whether a unital full amalgamated free
product of quasidiagonal C*-algebras is quasidiagonal again. We give a
sufficient condition such that a unital full amalgamated free product of
quasidiagonal C*-algebras with amalgamation over a finite dimensional
C*-algebra is quasidiagonal. Applying this result, we conclude that a unital
full free product of two AF algebras with amalgamation over a
finite-dimensional C*-algebra is AF if there are faithful tracial states on
each of these two AF algebras such that the restrictions on the common
subalgebra agree.

\end{abstract}
\maketitle

\section{Introduction}

Quasidiagonal C*-algebras have now been studied for more than 30 years. A set
$\Omega\subseteq\mathcal{B}\left(  \mathcal{H}\right)  $ is quasidiagonal if
for every $\varepsilon>0$ and finite subsets $\omega\subseteq\Omega$ and
$\chi\subseteq\mathcal{H}$ there is a finite-rank orthogonal projection $P$
such that $\left\Vert \left[  P,T\right]  \right\Vert <\varepsilon$ if
$T\in\omega$ and $\left\Vert \left(  1-P\right)  h\right\Vert <\varepsilon$ if
$h\in\chi$. In fact, if $\mathcal{A}$ is separable, then $\Omega$ is a
quasidiagonal set of operators if and only if there exist an increasing
sequence of finite rank projections, $P_{1}\leq P_{2}\leq\cdots,$ such that,
for each $a\in\Omega,$ $\left\Vert \left[  a,P_{n}\right]  \right\Vert
\rightarrow0$ and $P_{n}\rightarrow I_{\mathcal{H}}$ (strong operator
topology) as $n\rightarrow\mathcal{1}$. A C*-algebra $\mathcal{A}$ is
quasidiagonal (QD) if there is a faithful representation $\rho$ such that
$\rho\left(  \mathcal{A}\right)  $ is a quasidiagonal set of operators. Recall
that a faithful representation $\pi:\mathcal{A\rightarrow B}\left(
\mathcal{H}\right)  $ is called essential if $\pi\left(  \mathcal{A}\right)  $
contains no nonzero finite rank operators. Voiculescu showed that
$\mathcal{A}$ is QD if and only if $\pi\left(  \mathcal{A}\right)  $ is a
quasidiagonal set of operators for a faithful essential representation $\pi$
of $\mathcal{A}$. The examples of QD algebras include all abelian C*-algebras
and finite-dimensional C*-algebras as well as residually finite-dimensional
C*-algebras. For more information about QD C*-algebras, we refer the reader to
\cite{BN}, \cite{V}, \cite{Vo}.

All C*-algebras in this note are unital. In \cite{BK}, we know that all
separable QD C*-algebras are Blackadar and Kirchberg's MF algebras. It is well
known that the reduced free group C*-algebra $C_{r}^{\ast}\left(
F_{2}\right)  $ is not QD. Haagerup and Thorbj$\phi$rnsen showed that
$C_{r}^{\ast}\left(  F_{2}\right)  $ is MF (\cite{HT}). This implies that the
family of all separable QD C*-algebras are strictly contained in the set of MF C*-algebras.

In this note, we are interested in the question of whether the unital full
free products of QD C*-algebras with amalgamation over a common
finite-dimensional C*-algebra are QD again. In \cite{LS}, a necessary and
sufficient condition is given for a unital full free product of RFD
C*-algebras with amalgamation over a finite-dimensional C*-algebra to be RFD
again. Similar result hold for unital MF algebras (\cite{LS2}). Based on these
results and the relationship among RFD C*-algebras, MF C*-algebras and QD
C*-algebras, it is natural to ask whether the same things will happen when we
consider QD C*-algebras. For the case when the common part of two QD
C*-algebras in a unital full amalgamated free product is *-isomorphic to a
full matrix algebra (Proposition 1, \cite{LS}), the answer is affirmative. In
this note, we consider the case when the common part is a finite-dimensional
C*-algebra. First of all, we will give two corollaries about QD C*-algebras
based on Voiculescu's result. Then, we will show that, under a certain
condition, a unital full amalgamated free product of two separable QD
C*-algebras with amalgamation over a finite-dimensional C*-algebra is QD
again. As an application, we consider the case when two unital C*-algebras in
a unital full amalgamated free product are both AF. We will show that a full
free product of two AF algebras with amalgamation over a finite-dimensional
C*-subalgebra is a QD C*-algebra if there are faithful tracial states on each
of these two AF algebras such that the restrictions on the common subalgebra agree.

A brief overview of this paper is as follows. In Section 2, we fix some
notation and give two corollaries about QD C*-algebras based on Voiculescu's
abstract characterization. Section 3 is devoted to results on the full
amalgamated free products of two unital QD C*-algebras. We first consider
unital full free products of unital QD C*-algebras with amalgamation over
finite-dimensional C*-subalgebras. Then we consider the case when two QD
C*-algebras are both AF algebras.

\section{Unital QD C*-algebras}

The examples of QD algebras contains all RFD C*-algebras and AF algebras.
Voiculescu's result (Theorem 1, \cite{V}) give an abstract (i.e.
representation free) characterization of QD C*-algebras which is crucial in
the rest of this paper. In the rest of this paper, we will only be concerned
with separable C*-algebras and representations on separable Hilbert spaces.

We denote the set of all bounded operators on $\mathcal{H}$ by $\mathcal{B}%
\left(  \mathcal{H}\right)  $. Suppose $\left\{  x,x_{k}\right\}
_{k=1}^{\mathcal{1}}$ is a family of elements in $\mathcal{B}\left(
\mathcal{H}\right)  $. We say that $x_{k}\rightarrow x$ in $\ast$-s.o.t
($\ast$-strong operator topology) if and only if $x_{k}\rightarrow x$ in s.o.t
(strong operator topology) and $x_{k}^{\ast}\rightarrow x^{\ast}$ in s.o.t.

We use the notation C*$\left(  x_{1},x_{2,}\cdots\right)  $ to denote the
unital C*-algebra generated by $\left\{  x_{1},x_{2},\cdots\right\}  .$ Let
$\mathbb{C\langle}\mathbf{X}_{1},\ldots,\mathbf{X}_{n}\mathbb{\rangle}$ be the
set of all noncommutative polynomials in the indeterminants $\mathbf{X}%
_{1},\ldots,\mathbf{X}_{n}$. Let $\mathbb{C}_{\mathbb{Q}}=\mathbb{Q}%
+i\mathbb{Q}$ denote the complex-rational numbers, i.e., the numbers whose
real and imaginary parts are rational. Then the set $\mathbb{C}_{\mathbb{Q}%
}\mathbb{\langle}\mathbf{X}_{1},\ldots,\mathbf{X}_{n}\mathbb{\rangle}$ of
noncommutative polynomials with complex-rational coefficients is countable.
Throughout this paper we write%

\[
\mathbb{C\langle}\mathbf{X}_{1},\mathbf{X}_{2},\cdots\mathbb{\rangle=\cup
}_{m=1}^{\infty}\mathbb{C\langle}\mathbf{X}_{1},\mathbf{X}_{2},\cdots
\mathbf{X}_{m}\mathbb{\rangle}\text{,}%
\]
and
\[
\mathbb{C}_{\mathbb{Q}}\mathbb{\langle}\mathbf{X}_{1},\mathbf{X}_{2}%
,\cdots\mathbb{\rangle=\cup}_{m=1}^{\infty}\mathbb{C}_{\mathbb{Q}%
}\mathbb{\langle}\mathbf{X}_{1},\mathbf{X}_{2},\cdots\mathbf{X}_{m}%
\mathbb{\rangle}.
\]

Suppose $\{\mathcal{M}_{k_{n}}(\mathbb{C})\}_{n=1}^{\infty}$ is a sequence of
complex matrix algebras. We introduce the C*-direct product $\prod
_{m=1}^{\infty}\mathcal{M}_{k_{m}}(\mathbb{C)}$ of $\{\mathcal{M}_{k_{n}%
}(\mathbb{C})\}_{n=1}^{\infty}$ as follows:
\[
\prod_{n=1}^{\infty}\mathcal{M}_{k_{n}}(\mathbb{C})=\{(Y_{n})_{n=1}^{\infty
}\ |\ \forall\ n\geq1,\ Y_{n}\in\mathcal{M}_{k_{n}}(C)\ \text{ and
}\ \left\Vert (Y_{n})_{n=1}^{\infty}\right\Vert =\sup_{n\geq1}\Vert Y_{n}%
\Vert<\infty\}.
\]
Furthermore, we can introduce a norm-closed two sided ideal in $\prod
_{n=1}^{\infty}\mathcal{M}_{k_{n}}(\mathbb{C})$ as follows:
\[
\overset{\infty}{\underset{n=1}{\sum}}\mathcal{M}_{k_{n}}(\mathbb{C})=\left\{
\left(  Y_{n}\right)  _{n=1}^{\infty}\in\prod_{n=1}^{\infty}\mathcal{M}%
_{k_{n}}(\mathbb{C}):\lim\limits_{n\rightarrow\infty}\left\Vert Y_{n}%
\right\Vert =0\right\}  .
\]
Let $\pi$ be the quotient map from $\prod_{n=1}^{\infty}\mathcal{M}_{k_{n}%
}(\mathbb{C})$ to $\prod_{n=1}^{\infty}\mathcal{M}_{k_{n}}(\mathbb{C}%
)/\overset{\infty}{\underset{n=1}{\sum}}\mathcal{M}_{k_{n}}(\mathbb{C})$.
Then
\[
\prod_{n=1}^{\infty}\mathcal{M}_{k_{n}}(\mathbb{C})/\overset{\infty}%
{\underset{n=1}{\sum}}\mathcal{M}_{k_{n}}(\mathbb{C})
\]
is a unital C*-algebra. If we denote $\pi\left(  \left(  Y_{n}\right)
_{n=1}^{\infty}\right)  $ by $\left[  \left(  Y_{n}\right)  _{n}\right]  $,
then
\begin{equation}
\left\Vert \left[  \left(  Y_{n}\right)  _{n}\right]  \right\Vert
=\underset{n\rightarrow\mathcal{1}}{\lim\sup}\left\Vert Y_{n}\right\Vert
\leq\sup_{n}\left\Vert Y_{n}\right\Vert =\left\Vert \left(  Y_{n}\right)
_{n}\right\Vert \in\prod_{n=1}^{\infty}\mathcal{M}_{k_{n}}(\mathbb{C})
\tag{0}\label{0}%
\end{equation}

Recall that a C*-algebra is residually finite-dimensional (RFD) if it has a
separating family of finite-dimensional representations. If a separable C*-
algebra $\mathcal{A}$ can be embedded into $%
{\textstyle\prod\limits_{k}}
\mathcal{M}_{n_{k}}\left(  \mathbb{C}\right)  /\sum_{k}\mathcal{M}_{n_{k}%
}\left(  \mathbb{C}\right)  $ for a sequence of positive integers $\left\{
n_{k}\right\}  _{k=1}^{\mathcal{1}},$ then $\mathcal{A}$ is called an MF
algebra. Many properties of MF algebras were discussed in \cite{BK}. Note that
the family of all RFD C*-algebras is strictly contained in the family of all
QD C*-algebras, and all QD C*-algebras are MF C*-algebras. D. Hadwin give a
characterization of unital RFD C*-algebras (Theorem 6, \cite{H3}) and a
similar characterization of unital MF algebras is given by Hadwin, Li and Shen
in \cite{HLS}. Based on proceeding characterizations of RFD C*-algebras and MF
C*-algebras respectively, we are expecting to see the distinction of
quasidiagonal C*-algebras.

Now, we are ready to give a corollary of Voiculescu's result (Theorem 1,
\cite{V}).

\begin{corollary}
\label{55}Suppose $\mathcal{A}$ is a unital separable C*-algebra. Then
$\mathcal{A}$ is quasidiagonal if and only if, for any faithful unital
essential representation $\pi:\mathcal{A\rightarrow B\ }\left(  \mathcal{H}%
\right)  $ on a separable Hilbert space, there are sequences $\left\{
P_{n}\right\}  $ of finite-rank projections with $P_{n}\rightarrow I$ in
$s.o.t$ and unital completely positive mappings $\left\{  \varphi_{n}\right\}
$ from $\mathcal{A}$ into $\mathcal{B}\left(  P_{n}\mathcal{H}P_{n}\right)  $
such that
\[
\varphi_{n}\left(  a\right)  \rightarrow\pi\left(  a\right)  \text{ \ in s.o.t
for any }a\in\mathcal{A}%
\]
and%
\[
\left\Vert \varphi_{n}\left(  ab\right)  -\varphi_{n}\left(  a\right)
\varphi_{n}\left(  b\right)  \right\Vert \rightarrow0\text{ for any }%
a,b\in\mathcal{A}%
\]

\end{corollary}

\begin{proof}
$\left(  \Longleftarrow\right)  $Let $\left\{  a_{1},\cdots,a_{m}\right\}
\mathcal{\ }$be a finite subset of $\mathcal{A}$, $\left\{  x_{1},x_{2}%
,\cdots\right\}  $ be a dense subset of $\left(  \mathcal{H}\right)  _{1}.$
Suppose $\pi:\mathcal{A\rightarrow B}\left(  \mathcal{H}\right)  $ is a
faithful unital essential *-homomorphism, and $\left\{  \varphi_{n}\right\}  $
is a sequence of unital completely positive maps in the hypothesis. Then for
any $\varepsilon>0,$ there is an integer $n_{\varepsilon}$ and $x_{i_{1}%
},\cdots,x_{i_{m}}$ such that and%
\[
\left\Vert \pi\left(  a_{j}\right)  x_{i_{j}}\right\Vert >\left\Vert
a_{j}\right\Vert -\frac{\varepsilon}{2}%
\]
and%
\[
\left\Vert \varphi_{n_{\varepsilon}}\left(  a_{j}\right)  x_{i_{j}}-\pi\left(
a_{j}\right)  x_{i_{j}}\right\Vert <\frac{\varepsilon}{2}%
\]
as well as
\[
\left\Vert \varphi_{n_{\varepsilon}}\left(  ab\right)  -\varphi
_{n_{\varepsilon}}\left(  a\right)  \varphi_{n_{\varepsilon}}\left(  b\right)
\right\Vert <\varepsilon.
\]
This implies that $\left\Vert \varphi_{n_{\varepsilon}}\left(  a_{j}\right)
\right\Vert \geq\left\Vert a_{j}\right\Vert -\varepsilon$ $\left(
j=1,\cdots,m\right)  $ and $\left\Vert \varphi_{n_{\varepsilon}}\left(
ab\right)  -\varphi_{n_{\varepsilon}}\left(  a\right)  \varphi_{n_{\varepsilon
}}\left(  b\right)  \right\Vert <\varepsilon.$ Then by Theorem 1, \cite{V}, we
have $\pi\left(  \mathcal{A}\right)  $ is a QD algebra.

$\left(  \Rightarrow\right)  $Suppose $\mathcal{A}$ is a separable
quasidiagonal C*-algebra and $\pi:\mathcal{A\rightarrow B}\left(
\mathcal{H}\right)  $ is a faithful unital essential representation on a
separable Hilbert space. Then $\pi\left(  \mathcal{A}\right)  $ is a
quasidiagonal set. Therefore, we can find a sequence $\left\{  P_{n}\right\}
$ of projections with $P_{n}\rightarrow I$ in $s.o.t$ such that $P_{n}%
\pi\left(  a\right)  P_{n}\rightarrow\pi\left(  a\right)  $ in s.o.t. and
\[
\left\Vert P_{n}\pi\left(  ab\right)  P_{n}-P_{n}\pi\left(  a\right)
P\pi\left(  b\right)  P_{n}\right\Vert \rightarrow0.
\]
Now let $\varphi_{n}\left(  a\right)  =P_{n}\pi\left(  ab\right)  P_{n}$, then
the proof is completed.
\end{proof}

Since each separable QD C*-algebra $\mathcal{A}$ is MF, this implies that
$\mathcal{A}$ can be embedded into $%
{\displaystyle\prod}
\mathcal{M}_{k_{m}}\left(  \mathbb{C}\right)  /\sum$ $\mathcal{M}_{k_{m}%
}\left(  \mathbb{C}\right)  $ for a sequence $\left\{  k_{m}\right\}  $ of
integers. But in general, an MF algebra may not be a QD C*-algebra. Next
proposition shows the difference between QD C*-algebras and general MF
C*-algebras. Note that $%
{\displaystyle\prod}
\mathcal{M}_{k_{m}}\left(  \mathbb{C}\right)  $ can be viewed as a
C*-subalgebra of $\mathcal{B}\left(  \oplus_{m=1}^{\mathcal{1}}\mathbb{C}%
^{k_{m}}\right)  $. If $\mathcal{K}\left(  \oplus_{m=1}^{\mathcal{1}%
}\mathbb{C}^{k_{m}}\right)  $ denote the set of all compact operators acting
on $\oplus_{m=1}^{\mathcal{1}}\mathbb{C}^{k_{m}},$ then it is not hard to see
that $\mathcal{K}\left(  \oplus_{m=1}^{\mathcal{1}}\mathbb{C}^{k_{m}}\right)
\cap%
{\displaystyle\prod}
\mathcal{M}_{k_{m}}\left(  \mathbb{C}\right)  =\sum$ $\mathcal{M}_{k_{m}%
}\left(  \mathbb{C}\right)  .$

Let $\pi:\mathcal{B}\left(  \mathcal{H}\right)  \rightarrow\mathcal{Q}\left(
\mathcal{H}\right)  $ be the canonical mapping onto the Calkin algebra and
$\mathcal{A}$ is a unital C*-algebra. Suppose $\varphi:\mathcal{A\rightarrow
B}\left(  \mathcal{H}\right)  $ is a unital completely positive map then we
say that $\varphi$ is a representation modulo the compacts if $\pi\circ
\varphi:\mathcal{A\rightarrow Q}\left(  \mathcal{H}\right)  $ is a
*-homomorphism. If $\pi\circ\varphi$ is injective then we say that $\varphi$
is a faithful representation modulo the compacts. The following proposition
can be found in Proposition 3.1.3 and the preceding remark in \cite{BK}, we
include a proof for the convenience of the reader.

\begin{proposition}
\label{66}Suppose $\mathcal{A}$ is a unital separable C*-algebra,
$\mathcal{A}$ is QD if and only if there is a sequence $\left\{
k_{m}\right\}  $ of integers and an embedding $\rho$ from $\mathcal{A}$ into $%
{\displaystyle\prod}
\mathcal{M}_{k_{m}}\left(  \mathbb{C}\right)  /\sum$ $\mathcal{M}_{k_{m}%
}\left(  \mathbb{C}\right)  $ such that $\rho$ can be lifted to a faithful
representation of $\mathcal{A}$ into $%
{\displaystyle\prod}
\mathcal{M}_{k_{m}}\left(  \mathbb{C}\right)  $ modulo the compacts.
\end{proposition}

\begin{proof}
$\left(  \Longrightarrow\right)  $ Suppose $\mathcal{A}$ is a separable
quasidiagonal C*-algebra and $\pi:\mathcal{A\rightarrow B}\left(
\mathcal{H}\right)  $ is a faithful unital essential representation on a
separable Hilbert space. Then, for any given $\varepsilon>0,$ and
$\mathcal{F\subseteq A}$ finite subset, there is a representation
$\rho:\mathcal{A\rightarrow B}\left(  \mathcal{K}\right)  $ and a finite-rank
orthogonal projection $P$ such that
\[
\left\Vert P\rho\left(  a\right)  P\right\Vert \geq\left\Vert a\right\Vert
-\varepsilon
\]%
\[
\left\Vert \left[  P,\rho\left(  a\right)  \right]  \right\Vert \leq
\varepsilon
\]
by Theorem 1, \cite{V}. Let $1\in\mathcal{F}_{1}\subseteq\mathcal{F}%
_{2}\subseteq\cdots$be a sequence of finite subsets of $\mathcal{A}$ with
$\overline{\cup\mathcal{F}_{n}}=\mathcal{A}$. Assume that $\mathcal{F}%
_{n},P_{n},\mathcal{K}_{n},\varepsilon_{n}=n^{-2}$ and $\rho_{n}$ satisfy
above requirements. There is no loss of generality in assuming, we assume that
$\mathcal{K}_{n}$ is separable for each $n.$ Let%
\[
X=\oplus_{k\geq1}\mathcal{K}_{k}\text{, \ \ }\widetilde{\rho}=\oplus_{k\geq
1}\rho_{k}%
\]%
\[
Y=\oplus_{k\geq1}P_{k}\mathcal{K}_{k}\subseteq X
\]
Since, for each $a\in\cup\mathcal{F}_{n}$, there is $k_{o}\in\mathbb{N}$ such
that $\left\Vert \left[  P_{k},\rho\left(  a\right)  \right]  \right\Vert \leq
k^{-2}$ and $\left\Vert P_{k}\rho\left(  a\right)  P_{k}\right\Vert
\geq\left\Vert a\right\Vert -k^{-2}$ for each $k\geq k_{o},$ we have that
$P_{Y}\widetilde{\rho}|_{Y}$ is a faithful representations modulo the compacts
from $\mathcal{A}$ to $\mathcal{B}\left(  Y\right)  .$

Let $k_{m}=\dim P_{m}\mathcal{K}_{m}.$ Then $\mathcal{A}$ can be embedded into
$%
{\displaystyle\prod}
\mathcal{M}_{k_{m}}\left(  \mathbb{C}\right)  /\sum$ $\mathcal{M}_{k_{m}%
}\left(  \mathbb{C}\right)  $ such that this embedding can be lifted to a
faithful representation modulo the compacts.

$\left(  \Longleftarrow\right)  $ Assume that there is a sequence $\left\{
k_{m}\right\}  $ of integers and an embedding $\rho$ from $\mathcal{A}$ into $%
{\displaystyle\prod}
\mathcal{M}_{k_{m}}\left(  \mathbb{C}\right)  /\sum$ $\mathcal{M}_{k_{m}%
}\left(  \mathbb{C}\right)  $ such that $\rho$ can be lifted to a faithful
representation $\widetilde{\rho}$ of $\mathcal{A}$ into $%
{\displaystyle\prod}
\mathcal{M}_{k_{m}}\left(  \mathbb{C}\right)  \subseteq\mathcal{B}\left(
\oplus_{m=1}^{\mathcal{1}}\mathbb{C}^{k_{m}}\right)  $ modulo the compacts.
Let $P_{m}=P_{\mathbb{C}^{k_{m}}}\in\mathcal{B}\left(  \oplus_{m=1}%
^{\mathcal{1}}\mathbb{C}^{k_{m}}\right)  .$ Define a unital completely
positive map $\varphi_{n}=P_{n}\widetilde{\rho}P_{n}$ from $\mathcal{A}$ to
$\mathcal{B}\left(  \mathbb{C}^{k_{m}}\right)  =\mathcal{M}_{k_{m}}\left(
\mathbb{C}\right)  $ for each $m\in\mathbb{N}$. Since $\widetilde{\rho}$ is a
faithful representation modulo the compacts, we have $\left\Vert \varphi
_{n}\left(  ab\right)  -\varphi_{n}\left(  a\right)  \varphi_{n}\left(
b\right)  \right\Vert \rightarrow0$ as $n\rightarrow\mathcal{1}$. Suppose
$\mathcal{F\subseteq A}$ is a finite subset and $\varepsilon>0$. Assume
\[
\widetilde{\rho}\left(  a\right)  =\left(  a_{m}\right)  _{m=1}^{\mathcal{1}%
}\in%
{\displaystyle\prod}
\mathcal{M}_{k_{m}}\left(  \mathbb{C}\right)  \text{ for }a\in\mathcal{A}%
\text{.}%
\]
Then $\varphi_{n}(a)=P_{n}\rho\left(  a\right)  P_{n}=a_{n}.$ Since
$\underset{m\rightarrow\mathcal{1}}{\lim\sup}\left\Vert a_{m}\right\Vert
_{\mathcal{M}_{k_{m}}\left(  \mathbb{C}\right)  }=\left\Vert a\right\Vert ,$we
can find natural numbers $m_{1}$ and $m_{2}$ with $m_{1}\leq m_{2}$ such that
\[
\left\vert \left(  \sup_{m_{1}\leq l\leq m_{2}}\left\Vert a_{l}\right\Vert
_{\mathcal{M}_{l}\left(  \mathbb{C}\right)  }\right)  -\left\Vert a\right\Vert
\right\vert \leq\varepsilon\text{ for each }a\in\mathcal{F}%
\]
and
\[
\left\Vert \left(  \oplus_{m_{1}\leq l\leq m_{2}}\left(  ab\right)
_{l}\right)  -\oplus_{m_{1}\leq l\leq m_{2}}a_{l}b_{l}\right\Vert =\sup
_{m_{1}\leq l\leq m_{2}}\left\Vert \varphi_{l}\left(  ab\right)  -\varphi
_{l}\left(  a\right)  \varphi_{l}\left(  b\right)  \right\Vert \leq
\varepsilon\text{ for }a,b\in\mathcal{F}%
\]
Let $k=\sum_{l=m_{1}}^{m_{2}}l$ and $\varphi=\varphi_{m_{1}}\oplus\cdots
\oplus\varphi_{m_{2}}.$ Then $\varphi$ is a unital completely positive map
from $\mathcal{A}$ to $\mathcal{M}_{k_{m_{1}}}\left(  \mathbb{C}\right)
\oplus\cdots\oplus\mathcal{M}_{k_{m_{2}}}\left(  \mathbb{C}\right)  $ with
$\left\Vert \varphi\left(  a\right)  \right\Vert \geq\left\Vert a\right\Vert
-\varepsilon$ and $\left\Vert \varphi\left(  ab\right)  -\varphi\left(
a\right)  \varphi\left(  b\right)  \right\Vert \leq\varepsilon.$ This implies
that $\mathcal{A}$ is QD by Theorem 1, \cite{V}.
\end{proof}

\section{Unital Full Amalgamated Free Product of QD algebras}

Now we are ready to consider the unital full free products of two QD
C*-algebra with amalgamation over a finite-dimensional C*-subalgebra.

Given $\left(  \mathcal{A}_{i}\right)  _{i\in I\text{ }}$unital C*-algebras
with a common unital C*-subalgebra $\mathcal{B}$ and faithful conditional
expectation $E_{i}:\mathcal{A}_{i}\rightarrow\mathcal{B}$, look at the
algebraic free product $\mathcal{A=\circledast}_{\mathcal{B}}\mathcal{A}_{i}$
with amalgamation over $\mathcal{B}$, which is a $\mathcal{B}$-ring. Then
$\mathcal{B}$-bimodule decompositions $\mathcal{A}_{i}=\mathcal{B\oplus A}%
_{i}^{0}$ where $\mathcal{A}_{i}^{0}=\ker E_{i}=\left\{  a-E_{i}\left(
a\right)  :a\in\mathcal{A}_{i}\right\}  $ yield the following $\mathcal{B}%
$-bimodule decomposition (\cite{BF}):%
\[
\mathcal{\circledast}_{\mathcal{B}}\mathcal{A}_{i}=\mathcal{B\oplus\oplus
}_{i\neq\cdots\neq i_{n}:n\geq1}\mathcal{A}_{i_{1}}^{0}\otimes_{\mathcal{B}%
}\cdots\otimes_{\mathcal{B}}\mathcal{A}_{i_{n}}^{0}.
\]
The full amalgamated free product of $\left(  \mathcal{A}_{i},E_{i}\right)
_{i\in I}$ denoted by $\ast_{\mathcal{B}}\mathcal{A}_{i}$ is the completion of
$\mathcal{\circledast}_{\mathcal{B}}\mathcal{A}_{i}$ in the C*-norm.
\[
\left\Vert a\right\Vert =\sup\left\{  \left\Vert \pi\left(  a\right)
\right\Vert :\pi\text{ *-representation of }\mathcal{\circledast}%
_{\mathcal{B}}\mathcal{A}_{i}\right\}  .
\]

The following example shows that a full amalgamated free product of two QD (or
MF, RFD) algebras may not be QD (or MF, RFD) again, even for a unital full
free product of two full matrix algebras with amalgamation over a two
dimensional C*-algebra which is *-isomorphic to $\mathbb{C\oplus C}$.

\begin{example}
Let $\mathcal{D=}\mathbb{C\oplus C}$. Suppose $\varphi_{1}%
:\mathcal{D\rightarrow M}_{2}\left(  \mathbb{C}\right)  $ and $\varphi
_{2}:\mathcal{D\rightarrow M}_{3}\left(  \mathbb{C}\right)  $ are unital
embeddings such that
\[
\varphi_{1}\left(  1\oplus0\right)  =\left(
\begin{array}
[c]{cc}%
1 & 0\\
0 & 0
\end{array}
\right)  \text{ and }\varphi_{2}\left(  1\oplus0\right)  =\left(
\begin{array}
[c]{ccc}%
1 & 0 & 0\\
0 & 0 & 0\\
0 & 0 & 0
\end{array}
\right)
\]
Then $\mathcal{M}_{2}\left(  \mathbb{C}\right)  \underset{\mathcal{D}%
}{\mathcal{\ast}}\mathcal{M}_{3}\left(  \mathbb{C}\right)  $ is not QD. Note
that every QD algebra has a nontrivial tracial state by 2.4 \cite{Vo}. If we
assume that $\mathcal{M}_{2}\left(  \mathbb{C}\right)  \underset{\mathcal{D}%
}{\mathcal{\ast}}\mathcal{M}_{3}\left(  \mathbb{C}\right)  $ is QD, then there
exists a tracial state $\tau$ on $\mathcal{M}_{2}\left(  \mathbb{C}\right)
\underset{\mathcal{D}}{\mathcal{\ast}}\mathcal{M}_{3}\left(  \mathbb{C}%
\right)  .$ So the restrictions of $\tau$ on $\mathcal{M}_{2}\left(
\mathbb{C}\right)  $ and $\mathcal{M}_{3}\left(  \mathbb{C}\right)  $ are the
unique tracial states on $\mathcal{M}_{2}\left(  \mathbb{C}\right)  $ and
$\mathcal{M}_{3}\left(  \mathbb{C}\right)  $ respectively. It follows that
$\tau\left(  \varphi_{1}\left(  1\oplus0\right)  \right)  =\frac{1}{2}\neq$
$\tau$ $\left(  \varphi_{2}\left(  1\oplus0\right)  \right)  =\frac{1}{3}$
which contradicts to the fact that $\varphi_{1}\left(  1\oplus0\right)
=\varphi_{2}\left(  0\oplus1\right)  $ in $\mathcal{M}_{2}\left(
\mathbb{C}\right)  \underset{\mathcal{D}}{\mathcal{\ast}}\mathcal{M}%
_{3}\left(  \mathbb{C}\right)  $. Therefore $\mathcal{M}_{2}\left(
\mathbb{C}\right)  \underset{\mathcal{D}}{\mathcal{\ast}}\mathcal{M}%
_{3}\left(  \mathbb{C}\right)  $ is not QD.
\end{example}

The following lemma is well known.

\begin{lemma}
\label{13.2}Suppose $\mathcal{A=}$C*$\left(  x_{1},x_{2},\cdots\right)  $ and
$\mathcal{B=}$C*$\left(  y_{1},y_{2},\cdots\right)  $ are unital C*-algebras.
Then there is a unital *-homomorphism from $\mathcal{A}$ to $\mathcal{B}$
sending each $x_{k}$ to $y_{k},$ if and only if, for each $\ast$-polynomial
$P\in\mathbb{C}_{\mathbb{Q}}\left\langle \mathbf{X}_{1},\mathbf{X}_{2}%
,\cdots\right\rangle ,$ we have%
\[
\left\Vert P\left(  x_{1},x_{2},\cdots\right)  \right\Vert \geq\left\Vert
P\left(  y_{1},y_{2},\cdots\right)  \right\Vert .
\]

\end{lemma}

\begin{lemma}
\label{5}(Corollary 4, \cite{BF}) Given $\left(  \mathcal{A}_{i},E_{i}\right)
_{i\in I}$ and $\left(  \mathcal{B}_{i},F_{i}\right)  _{i\in I}$ with
$E_{i}:\mathcal{A}_{i}\rightarrow\mathcal{B}$, $F_{i}:\mathcal{B}%
_{i}\rightarrow\mathcal{B}$ faithful projections of norm one onto the unital
C*-subalgebra $\mathcal{B}$ and the $\mathcal{B}$-linear completely positive
maps $\varphi_{i}:\mathcal{A}_{i}\rightarrow\mathcal{B}_{i},$ there is a
common extension $\Phi:\ast_{\mathcal{B}}\mathcal{A}_{i}\rightarrow
\ast_{\mathcal{B}}\mathcal{B}_{i}$ which is $\mathcal{B}$-linear and
completely positive.
\end{lemma}

For showing our main result in this section, we need the following two lemmas.

\begin{lemma}
\label{4}(Proposition 2.2, \cite{ADRL})Let
\[
\widetilde{\mathcal{A}}\supseteq\mathcal{A\supseteq D\subseteq B\subseteq
}\widetilde{\mathcal{B}}%
\]
be inclusions of C*-algebras and let $\mathcal{A\ast}_{\mathcal{D}}%
\mathcal{B}$ and $\widetilde{\mathcal{A}}\mathcal{\ast}_{\mathcal{D}%
}\widetilde{\mathcal{B}}$ be the corresponding full amalgamated free product
C*-algebras. Let $\lambda:\mathcal{A\ast}_{\mathcal{D}}\mathcal{B}%
\rightarrow\widetilde{\mathcal{A}}\mathcal{\ast}_{\mathcal{D}}\widetilde
{\mathcal{B}}$ be the *-homomorphism arising via the universal property form
the inclusions $\mathcal{A\hookrightarrow}\widetilde{\mathcal{A}}$ and
$\mathcal{B\hookrightarrow}\widetilde{\mathcal{B}}$. Then $\lambda$ is injective.
\end{lemma}

\begin{lemma}
\label{7}(Corollary 2, \cite{LS}) Suppose that $\mathcal{A}$ is a separable
unital RFD C*-algebra and $\mathcal{D}$ is a unital finite-dimensional
C*-subalgebra of $\mathcal{A}.$ Then $\mathcal{A\ast}_{\mathcal{D}}%
\mathcal{A}$ is RFD.
\end{lemma}

\begin{theorem}
\label{8}Suppose C*-algebras $\mathcal{A}_{1}$ and $\mathcal{A}_{2}$ are
unital QD algebras and $\mathcal{D}$ is a common unital finite-dimensional
C*-subalgebra. If there is a sequence $\left\{  k_{n}\right\}  $ of integers
such that $\mathcal{A}_{1}$ and $\mathcal{A}_{2}$ can be both unital embedding
into $%
{\displaystyle\prod}
\mathcal{M}_{k_{m}}\left(  \mathbb{C}\right)  /\sum\mathcal{M}_{k_{m}}\left(
\mathbb{C}\right)  $, and these two unital embeddings can be lifted to
faithful representations
\[
q_{\mathcal{A}_{1}}:\mathcal{A}_{1}\mathcal{\rightarrow}%
{\displaystyle\prod}
\mathcal{M}_{k_{m}}\left(  \mathbb{C}\right)  \text{ and }q_{\mathcal{A}_{2}%
}:\mathcal{A}_{2}\mathcal{\rightarrow}%
{\displaystyle\prod}
\mathcal{M}_{k_{m}}\left(  \mathbb{C}\right)
\]
modulo the compacts respectively satisfying $q_{\mathcal{A}_{1}}$,
$q_{\mathcal{A}_{2}}$ agree on $\mathcal{D}$, $q_{\mathcal{A}_{i}%
}|_{\mathcal{D}}$ is a unital faithful *-homomorphism and $q_{\mathcal{A}_{1}%
},q_{\mathcal{A}_{2}}$ are $\mathcal{D}$-linear in the sense that
$q_{\mathcal{A}_{i}}\left(  dx\right)  =q_{\mathcal{A}_{i}}\left(  d\right)
q_{\mathcal{A}_{i}}\left(  x\right)  $ for $d\in\mathcal{D}$ and
$x\in\mathcal{A}_{i}$ $\left(  i=1\text{ or }2\right)  $, then $\mathcal{A}%
_{1}\mathcal{\ast}_{\mathcal{D}}\mathcal{A}_{2}$ is QD.
\end{theorem}

\begin{proof}
Suppose $E_{i}:\mathcal{A}_{i}\rightarrow\mathcal{D}$ $\left(  i=1,2\right)  $
is a faithful conditional expectation, let $\mathcal{A}_{i}=\mathcal{D\oplus
A}_{i}^{0}$ be the $\mathcal{D}$-bimodule decomposition of $\mathcal{A}_{i}$
with respect to $E_{i}.$

Assume $\mathcal{F}$ is a finite subset of $\mathcal{A}_{1}\mathcal{\ast
}_{\mathcal{D}}\mathcal{A}_{2}$ which contains only two elements. Then
$\mathcal{F=}$ $\left\{  d_{1}+a_{i_{1}}\cdots a_{i_{s}},d_{2}+b_{k_{1}}\cdots
b_{k_{t}}\right\}  $ where $d_{1},d_{2}\in\mathcal{D}$ and $a_{i_{h}}%
\in\mathcal{A}_{i_{h}}^{0},b_{k_{l}}\in\mathcal{A}_{k_{l}}^{0}$ as well as
$i_{1}\neq\cdots\neq i_{s},$ $k_{1}\neq\cdots\neq k_{t}.$ Since $a_{i_{k}}$
and $b_{k_{t}}$ are all in $\mathcal{A}_{1}$ or $\mathcal{A}_{2}$ for each
$i_{k}\in\left\{  i_{1},\cdots,i_{s}\right\}  $ and $k_{t}\in\left\{
k_{1,}\cdots,k_{t}\right\}  ,$without loss of generality, we may assume that
either
\[
\mathcal{F=}\left\{  d_{1}+a_{1,1}a_{2,1}a_{1,2}a_{2,2}a_{1,3}\cdots
a_{1,n},d_{2}+b_{1,1}b_{2,1}b_{1,2}b_{2,2}b_{1,3}\cdots b_{1,m}\right\}
\]
or
\[
\mathcal{F=}\left\{  d_{1}+a_{1,1}a_{2,1}a_{1,2}a_{2,2}a_{1,3}\cdots
a_{1,n},d_{2}+b_{2,1}b_{1,1}b_{2,2}b_{1,2}b_{2,3}\cdots b_{2,m}\right\}
\]
where
\[
\left\{  a_{1,1},a_{1,2},\cdots,a_{1,n},b_{1,1,}b_{1,2},\cdots,b_{1,m}%
\right\}  \subseteq\mathcal{A}_{1}^{0}%
\]
and
\[
\left\{  a_{2,1},a_{2,2},\cdots,a_{2,n},b_{2,1,}b_{2,2},\cdots,b_{2,m}%
\right\}  \subseteq\mathcal{A}_{2}^{0}.
\]
For the case when
\[
\mathcal{F=}\left\{  d_{1}+a_{1,1}a_{2,1}a_{1,2}a_{2,2}a_{1,3}\cdots
a_{1,n},d_{2}+b_{1,1}b_{2,1}b_{1,2}b_{2,2}b_{1,3}\cdots b_{1,m}\right\}  ,
\]
we have $a_{1,n}\cdot b_{11}$ and $b_{1,m}\cdot a_{11}$ are both in
$\mathcal{A}_{1}.$ So we can find an integer $N_{0}$ such that
\[
\widetilde{q}_{\mathcal{A}_{1}}:\mathcal{A}_{1}\mathcal{\rightarrow}%
{\displaystyle\prod_{k_{m}\geq k_{N_{0}}}}
\mathcal{M}_{k_{m}}\left(  \mathbb{C}\right)  \text{, }\widetilde
{q}_{\mathcal{A}_{2}}:\mathcal{A}_{2}\mathcal{\rightarrow}%
{\displaystyle\prod_{k_{m}\geq k_{N_{0}}}}
\mathcal{M}_{k_{m}}\left(  \mathbb{C}\right)
\]
are faithful representations modulo the compacts respectively with%
\begin{equation}
\left\Vert \widetilde{q}_{\mathcal{A}_{i}}\left(  a_{i,j}\right)  \right\Vert
\leq\left\Vert a_{i,j}\right\Vert +1,\left\Vert \widetilde{q}_{\mathcal{A}%
_{1}}\left(  b_{k,l}\right)  \right\Vert \leq\left\Vert b_{k,l}\right\Vert
+1\text{ \ \ \ \ \ \ }\tag{2}\label{a}%
\end{equation}
where $i,k\in\left\{  1,2\right\}  ,1\leq j\leq n$ and $1\leq l\leq m.$
Meanwhile, we also require that
\begin{align}
&  \left\Vert \widetilde{q}_{\mathcal{A}_{1}}\left(  a_{1,n}b_{11}\right)
-\widetilde{q}_{\mathcal{A}_{1}}\left(  a_{1,n}\right)  \widetilde
{q}_{\mathcal{A}_{1}}\left(  b_{1,1}\right)  \right\Vert \text{ and
}\left\Vert \widetilde{q}_{\mathcal{A}_{1}}\left(  b_{1,m}a_{1,1}\right)
-\widetilde{q}_{\mathcal{A}_{1}}\left(  b_{1,m}\right)  \widetilde
{q}_{\mathcal{A}_{1}}\left(  a_{1,1}\right)  \right\Vert \text{\ }%
\tag{3}\label{b}\\
&  \leq\frac{\varepsilon}{\left(  \left\Vert a_{1,1}\right\Vert +1\right)
\left(  \left\Vert a_{2,1}\right\Vert +1\right)  \cdots\left(  \left\Vert
a_{1,n}\right\Vert +1\right)  \left(  \left\Vert b_{1,1}\right\Vert +1\right)
\left(  \left\Vert b_{2,1}\right\Vert +1\right)  \cdots\left(  \left\Vert
b_{1,m}\right\Vert +1\right)  }\text{\ .\ \ }\nonumber
\end{align}

Since $\widetilde{q}_{\mathcal{A}_{1}}|_{\mathcal{D}}=\widetilde
{q}_{\mathcal{A}_{2}}|_{\mathcal{D}}$ is a unital faithful *-homomorphism, by
Lemma \ref{5}, we have a unital completely positive map
\[
\Phi=\widetilde{q}_{\mathcal{A}_{1}}\ast\widetilde{q}_{\mathcal{A}_{2}%
}:\mathcal{A}_{1}\mathcal{\ast}_{\mathcal{D}}\mathcal{A}_{2}%
\mathcal{\rightarrow}%
{\displaystyle\prod_{k_{m}\geq k_{N_{0}}}}
\mathcal{M}_{k_{m}}\left(  \mathbb{C}\right)  \ast_{\mathcal{D}}%
{\displaystyle\prod_{k_{m}\geq k_{N_{0}}}}
\mathcal{M}_{k_{m}}\left(  \mathbb{C}\right)
\]
with $\Phi|_{\mathcal{A}_{1}}=\widetilde{q}_{\mathcal{A}_{1}},$ $\Phi
|_{\mathcal{A}_{2}}=\widetilde{q}_{\mathcal{A}_{2}}.$

Suppose
\[
\mathcal{A}_{1}=C^{\ast}\left(  z_{1},\cdots,z_{l},x_{1},\cdots,a_{1,1}%
,\cdots,a_{1,n},b_{1,1},\cdots,b_{1,m}\right)
\]
and
\[
\mathcal{A}_{2}\mathcal{=}C^{\ast}\left(  z_{1},\cdots,z_{l},y_{1}%
,\cdots,a_{2,1},\cdots,a_{2,n-1},b_{2,1},\cdots b_{2,m-1}\right)
\]
where $\mathcal{D=}C^{\ast}\left(  z_{1},\cdots,z_{l}\right)  .$ Assume that
$X_{i}\left(  N_{0}\right)  =\widetilde{q}_{\mathcal{A}_{1}}\left(
x_{i}\right)  ,Y_{i}\left(  N_{0}\right)  =\widetilde{q}_{\mathcal{A}_{2}%
}\left(  y_{i}\right)  $ for each $i\in\mathbb{N}$, $D_{j}\left(
N_{0}\right)  =\widetilde{q}_{\mathcal{A}_{1}}\left(  z_{j}\right)
=\widetilde{q}_{\mathcal{A}_{2}}\left(  z_{j}\right)  $ for $1\leq j\leq l$
and
\[
A_{1,1}\left(  N_{0}\right)  =q_{\mathcal{A}_{1}}\left(  a_{1,1}\right)
,\cdots,A_{1,n}\left(  N_{0}\right)  =q_{\mathcal{A}_{1}}\left(
a_{1,n}\right)  ,
\]%
\[
A_{2,1}\left(  N_{0}\right)  =q_{\mathcal{A}_{2}}\left(  a_{2,1}\right)
,\cdots,A_{2,n-1}\left(  N_{0}\right)  =q_{\mathcal{A}_{2}}\left(
a_{2,n-1}\right)
\]
as well as
\[
B_{1,1}\left(  N_{0}\right)  =q_{\mathcal{A}_{1}}\left(  b_{1,1}\right)
,\cdots,B_{1,m}\left(  N_{0}\right)  =q_{\mathcal{A}_{1}}\left(
b_{1,m}\right)  ,
\]%
\[
B_{2,1}\left(  N_{0}\right)  =q_{\mathcal{A}_{2}}\left(  b_{2,1}\right)
,\cdots,B_{2,m-1}\left(  N_{0}\right)  =q_{\mathcal{A}_{2}}\left(
b_{2,m-1}\right)  .
\]
Then let
\begin{align*}
&  \mathcal{A}_{1}^{N_{0}}\\
&  =C^{\ast}\left(  \left\{  D_{m}\left(  N_{0}\right)  \right\}  _{m=1}%
^{l},\left\{  X_{i}\left(  N_{0}\right)  \right\}  _{i=1}^{\mathcal{1}%
},\left\{  A_{1,1}\left(  N_{0}\right)  ,\cdots,A_{1,n}\left(  N_{0}\right)
\right\}  ,\left\{  B_{1,1}\left(  N_{0}\right)  ,\cdots,B_{1,m}\left(
N_{0}\right)  \right\}  \right)
\end{align*}
and
\begin{align*}
&  \mathcal{A}_{2}^{N_{0}}\\
&  =C^{\ast}\left(  \left\{  D_{m}\left(  N_{0}\right)  \right\}  _{m=1}%
^{l},\left\{  Y\left(  N_{0}\right)  \right\}  _{i=1}^{\mathcal{1}},\left\{
A_{2,1}\left(  N_{0}\right)  ,\cdots,A_{2,n-1}\left(  N_{0}\right)  \right\}
,\left\{  B_{2,1}\left(  N_{0}\right)  ,\cdots,B_{2,m-1}\left(  N_{0}\right)
\right\}  \right)
\end{align*}
It is clear that $\mathcal{A}_{1}^{N_{0}}$ and $\mathcal{A}_{2}^{N_{0}}$ are
unital C*-subalgebras of $%
{\displaystyle\prod_{k_{m}\geq k_{N_{0}}}}
\mathcal{M}_{k_{m}}\left(  \mathbb{C}\right)  .$ So, by Lemma \ref{13.2} and
inequality (\ref{0}), there is a *-homomorphism
\[
\pi_{1}^{N_{0}}:\mathcal{A}_{1}^{N_{0}}\rightarrow\mathcal{A}_{1}%
\ast_{\mathcal{D}}\mathcal{A}_{2}%
\]
with $\pi_{1}^{N_{0}}\left(  C^{\ast}\left(  D_{1}\left(  N_{0}\right)
,\cdots,D_{l}\left(  N_{0}\right)  \right)  \right)  =\mathcal{D}$ and a
*-homomorphism%
\[
\pi_{2}^{N_{0}}:\mathcal{A}_{2}^{N_{0}}\rightarrow\mathcal{A}_{1}%
\ast_{\mathcal{D}}\mathcal{A}_{2}%
\]
with $\pi_{1}^{N_{0}}|_{C^{\ast}\left(  D_{1}\left(  N_{0}\right)
,\cdots,D_{l}\left(  N_{0}\right)  \right)  }=\pi_{2}^{N_{0}}|_{C^{\ast
}\left(  D_{1}\left(  N_{0}\right)  ,\cdots,D_{l}\left(  N_{0}\right)
\right)  }.$ Therefore, we have a homomorphism
\[
\pi^{N_{0}}:\mathcal{A}_{1}^{N_{0}}\ast_{\mathcal{D}}\mathcal{A}_{2}^{N_{0}%
}\rightarrow\mathcal{A}_{1}\ast_{\mathcal{D}}\mathcal{A}_{2}%
\]
such that $\pi^{N_{0}}|_{\mathcal{A}_{1}^{N_{0}}}=\pi_{1}^{N_{0}}$ and
$\pi^{N_{0}}|_{\mathcal{A}_{2}^{N_{0}}}=\pi_{2}^{N_{0}}.$ By Lemma \ref{4}, we
may treat $\mathcal{A}_{1}^{N_{0}}\ast_{\mathcal{D}}\mathcal{A}_{2}^{N_{0}}$
as a C*-subalgebra of $%
{\displaystyle\prod_{k_{m}\geq k_{N_{0}}}}
\mathcal{M}_{k_{m}}\left(  \mathbb{C}\right)  \ast_{\mathcal{D}}%
{\displaystyle\prod_{k_{m}\geq k_{N_{0}}}}
\mathcal{M}_{k_{m}}\left(  \mathbb{C}\right)  .$ Therefore
\begin{align*}
\pi^{N_{0}}\left(  \Phi\left(  d_{1}+a_{1,1}a_{2,1}\cdots a_{1,n}\right)
\right)   &  =\pi^{N_{0}}\left(  q_{\mathcal{A}_{1}}\left(  d_{1}\right)
+q_{\mathcal{A}_{1}}\left(  a_{1,1}\right)  q_{\mathcal{A}_{2}}\left(
a_{2,1}\right)  \cdots q_{\mathcal{A}_{1}}\left(  a_{1,n}\right)  \right)  \\
&  =d_{1}+a_{1,1}a_{2,1}\cdots a_{1,n}.
\end{align*}
It follows that
\begin{equation}
\left\Vert \Phi\left(  d_{1}+a_{1,1}a_{2,1}\cdots a_{1,n}\right)  \right\Vert
\geq\left\Vert d_{1}+a_{1,1}a_{2,1}\cdots a_{1,n}\right\Vert ,\tag{4}\label{c}%
\end{equation}
similarly,
\begin{equation}
\left\Vert \Phi\left(  d_{2}+b_{1,1}b_{2,1}\cdots b_{1,m}\right)  \right\Vert
\geq\left\Vert d_{2}+b_{1,1}b_{2,1}\cdots b_{1,m}\right\Vert .\tag{5}\label{d}%
\end{equation}

From Lemma \ref{7}, we know that $%
{\displaystyle\prod_{k_{m}\geq k_{N_{0}}}}
\mathcal{M}_{k_{m}}\left(  \mathbb{C}\right)  \ast_{\mathcal{D}}%
{\displaystyle\prod_{k_{m}\geq k_{N_{0}}}}
\mathcal{M}_{k_{m}}\left(  \mathbb{C}\right)  $ is RFD, then it can be unital
embedded into $%
{\displaystyle\prod}
\mathcal{M}_{l_{m}}$ for a sequence $\left\{  l_{m}\right\}  $ of integers. So
by Theorem 6, \cite{H3}, there is a sequence of finite-rank projection
$\left\{  P_{n}\right\}  $ on $\oplus_{m=1}^{\mathcal{1}}\mathbb{C}^{l_{m}}$
and *-representation
\[
\varphi_{n}:%
{\displaystyle\prod_{k_{m}\geq k_{N_{0}}}}
\mathcal{M}_{k_{m}}\left(  \mathbb{C}\right)  \ast_{\mathcal{D}}%
{\displaystyle\prod_{k_{m}\geq k_{N_{0}}}}
\mathcal{M}_{k_{m}}\left(  \mathbb{C}\right)  \rightarrow\mathcal{B}\left(
P_{n}\left(  \oplus_{m=1}^{\mathcal{1}}\mathbb{C}^{l_{m}}\right)
P_{n}\right)
\]
such that
\[
\varphi_{n}\left(  a\right)  \rightarrow a\text{ }\left(  s.o.t\right)  \text{
for each }a\in%
{\displaystyle\prod_{k_{m}\geq k_{N_{0}}}}
\mathcal{M}_{k_{m}}\left(  \mathbb{C}\right)  \ast_{\mathcal{D}}%
{\displaystyle\prod_{k_{m}\geq k_{N_{0}}}}
\mathcal{M}_{k_{m}}\left(  \mathbb{C}\right)  .
\]
So, for $\Phi\left(  d_{1}+a_{1,1}a_{2,1}\cdots a_{1,n}\right)  $ and
$\Phi\left(  d_{2}+b_{1,1}b_{2,1}\cdots b_{1,m}\right)  $ in $%
{\displaystyle\prod_{k_{m}\geq k_{N_{0}}}}
\mathcal{M}_{k_{m}}\left(  \mathbb{C}\right)  \ast_{\mathcal{D}}%
{\displaystyle\prod_{k_{m}\geq k_{N_{0}}}}
\mathcal{M}_{k_{m}}\left(  \mathbb{C}\right)  ,$ there is $M,$ such that
\begin{align*}
\left\Vert \varphi_{M}\left(  \Phi\left(  d_{1}+a_{1,1}a_{2,1}\cdots
a_{1,n}\right)  \right)  \right\Vert  &  \geq\left\Vert \Phi\left(
d_{1}+a_{1,1}a_{2,1}\cdots a_{1,n}\right)  \right\Vert -\varepsilon\\
&  \geq\left\Vert d_{1}+a_{i_{1}}\cdots a_{i_{s}}\right\Vert -\varepsilon
\end{align*}
by (\ref{c}) and%
\begin{align*}
\left\Vert \varphi_{M}\left(  \Phi\left(  d_{2}+b_{1,1}b_{2,1}\cdots
b_{1,m}\right)  \right)  \right\Vert  &  \geq\left\Vert \Phi\left(
d_{2}+b_{1,1}b_{2,1}\cdots b_{1,m}\right)  \right\Vert -\varepsilon\\
&  \geq\left\Vert d_{2}+b_{1,1}b_{2,1}\cdots b_{1,m}\right\Vert -\varepsilon
\end{align*}
by (\ref{d}). Meanwhile, note that $q_{\mathcal{A}_{i}}$ is $\mathcal{D}%
$-linear and by inequalities (\ref{a}) and (\ref{b}), we have
\begin{align*}
&  \left\Vert \varphi_{M}\left(  \Phi\left(  \left(  d_{2}+a_{1,1}\cdots
a_{1,n}\right)  \left(  d_{1}+b_{1,1}\cdots b_{1,m}\right)  \right)  \right)
-\varphi_{M}\left(  \Phi\left(  d_{2}+a_{1,1}\cdots a_{1,n}\right)
\Phi\left(  d_{1}+b_{1,1}\cdots b_{1,m}\right)  \right)  \right\Vert \\
&  \leq\left\Vert \widetilde{q}_{\mathcal{A}_{1}}\left(  a_{1,1}\right)
\cdots\widetilde{q}_{\mathcal{A}_{2}}\left(  a_{2,n-1}\right)  \right\Vert
\left\Vert \widetilde{q}_{\mathcal{A}_{1}}\left(  a_{1,n}b_{1,1}\right)
-\widetilde{q}_{\mathcal{A}_{1}}\left(  a_{1,n}\right)  \widetilde
{q}_{\mathcal{A}_{1}}\left(  b_{1,1}\right)  \right\Vert \left\Vert
\widetilde{q}_{\mathcal{A}_{2}}\left(  b_{2,1}\right)  \cdots\widetilde
{q}_{\mathcal{A}_{1}}\left(  b_{1,m}\right)  \right\Vert \\
&  \leq\varepsilon
\end{align*}
and
\begin{align*}
&  \left\Vert \varphi_{M}\left(  \Phi\left(  \left(  d_{2}+b_{1,1}\cdots
b_{1,m}\right)  \left(  d_{1}+a_{1,1}\cdots a_{1,n}\right)  \right)  \right)
-\varphi_{M}\left(  \Phi\left(  d_{2}+b_{1,1}\cdots b_{1,m}\right)
\Phi\left(  d_{1}+a_{1,1}\cdots a_{1,n}\right)  \right)  \right\Vert \\
&  \leq\varepsilon.
\end{align*}
For the case%
\[
\mathcal{F=}\left\{  d_{1}+a_{1,1}a_{2,1}a_{1,2}a_{2,2}a_{1,3}\cdots
a_{1,n},d_{2}+b_{2,1}b_{1,1}b_{2,2}b_{1,2}b_{2,3}\cdots b_{2,m}\right\}  ,
\]
we can use a similar discussion and notice that%
\[
\varphi_{M}\left(  \Phi\left(  \left(  d_{1}+a_{1,1}\cdots a_{1,n}\right)
\left(  d_{2}+b_{2,1}\cdots b_{2,m}\right)  \right)  \right)  =\varphi
_{M}\left(  \Phi\left(  d_{1}+a_{1,1}\cdots a_{1,n}\right)  \Phi\left(
d_{2}+b_{2,1}\cdots b_{2,m}\right)  \right)
\]
and%
\[
\varphi_{M}\left(  \Phi\left(  \left(  d_{2}+b_{2,1}\cdots b_{2,m}\right)
\left(  d_{1}+a_{1,1}\cdots a_{1,n}\right)  \right)  \right)  =\varphi
_{M}\left(  \Phi\left(  d_{2}+b_{2,1}\cdots b_{2,m}\right)  \Phi\left(
d_{1}+a_{1,1}\cdots a_{1,n}\right)  \right)
\]
in this case. It follows that the map $\Phi_{M}=\varphi_{M}\circ\Phi$ is a
unital completely positive mapping from $\mathcal{A}_{1}\mathcal{\ast
}_{\mathcal{D}}\mathcal{A}_{2}$ to $\mathcal{B}\left(  P_{n}\left(
\oplus_{m=1}^{\mathcal{1}}\mathbb{C}^{l_{m}}\right)  P_{n}\right)  $ such
that
\[
\left\Vert \Phi_{M}\left(  a\right)  \right\Vert \geq\left\Vert a\right\Vert
-\varepsilon\text{ and }\left\Vert \Phi_{M}\left(  ab\right)  -\Phi_{M}\left(
a\right)  \Phi_{M}\left(  b\right)  \right\Vert \leq\varepsilon
\]
for $a,b\in\mathcal{F}$. Using a similar argument for any finite subset
$\mathcal{F\subseteq}$ $\mathcal{A}_{1}\mathcal{\ast}_{\mathcal{D}}%
\mathcal{A}_{2},$ we conclude that $\mathcal{A}_{1}\mathcal{\ast}%
_{\mathcal{D}}\mathcal{A}_{2}$ is QD by Theorem 1, \cite{V}.
\end{proof}

For showing the following corollary, we need a lemma.

\begin{lemma}
(Theorem III.3.4, \cite{D})\label{15}A C*-algebra $\mathcal{A}$ is AF if and
only if it is separable and :

(*) for all $\varepsilon>0$ and $A_{1},\cdots,A_{n}$ in $\mathcal{A}$, there
exists a finite dimensional C*-subalgebra $\mathcal{B}$ of $\mathcal{A}$ such
that $dist(\mathcal{A}_{i},\mathcal{B)<\varepsilon}$ for $1\leq i\leq n.$

Moreover, if $\mathcal{A}_{1}$ is a finite-dimensional subalgebra of
$\mathcal{A}$, then we may choose $\mathcal{B}$ so that it contains
$\mathcal{A}_{1}.$
\end{lemma}

\begin{corollary}
\label{9}Suppose $\mathcal{A}$ and $\mathcal{B}$ are both AF algebras,
$\mathcal{D}$ is a common unital finite-dimensional C*-subalgebra of
$\mathcal{A}$ and $\mathcal{B}$. If there are faithful tracial states
$\tau_{\mathcal{A}}$ and $\tau_{\mathcal{B}}$ on $\mathcal{A}$ and
$\mathcal{B}$ respectively, such that%
\[
\tau_{\mathcal{A}}\left(  x\right)  =\tau_{\mathcal{B}}\left(  x\right)
\text{, \ \ }\mathcal{\forall}x\in\mathcal{D}\text{,}%
\]
then $\mathcal{A\ast}_{\mathcal{D}}\mathcal{B}$ is QD.
\end{corollary}

\begin{proof}
Assume that $\{x_{n}\}_{n=1}^{\infty}\subseteq\mathcal{A}$, $\{y_{n}%
\}_{n=1}^{\infty}\subseteq\mathcal{B}$ are families of generators in
$\mathcal{A}$ and $\mathcal{B}$ respectively. Note that $\mathcal{A}$ and
$\mathcal{B}$ are AF algebras, $\mathcal{D}$ is a finite-dimensional
subalgebra. For each $N\in\mathbb{N}$, there are finite dimensional C$^{\ast}%
$-subalgebras $\mathcal{A}_{N}\subseteq\mathcal{A}$ and $\mathcal{B}%
_{N}\subseteq\mathcal{B}$ such that
\begin{equation}
\max_{1\leq n\leq N}\{dist(x_{n},\mathcal{A}_{N}),dist(y_{n},\mathcal{B}%
_{N})\}\leq\frac{1}{N} \tag{6}\label{e}%
\end{equation}
and $\mathcal{A}_{N}\supset\mathcal{D}\subset\mathcal{B}_{N}$ by Lemma
\ref{15}. Note that $\tau_{\mathcal{A}}(x)=\tau_{\mathcal{B}}(x),\ \forall
\ x\in\mathcal{D}.$ From the argument in the proof of Theorem 4.2 \cite{ADRL},
there are rational faithful tracial states on $\mathcal{A}_{N}$ and
$\mathcal{B}_{N}$ such that their restrictions on $\mathcal{D}$ agree. This
implies that there is a positive integer $k_{N}$ such that
\[
\mathcal{M}_{k_{N}}(\mathbb{C)\supseteq\mathcal{A}}_{N}\mathbb{\supseteq
\mathcal{D}\subseteq\mathcal{B}}_{N}\mathbb{\subseteq\mathcal{M}}_{k_{N}%
}\mathbb{(C)}.
\]
So there are conditional expectations $E_{\mathcal{A}}^{N}%
:\mathcal{A\rightarrow A}_{N}$ and $E_{\mathcal{B}}^{N}:\mathcal{B\rightarrow
B}_{N}$such that $E_{\mathcal{A}}^{N}\left(  x\right)  =E_{\mathcal{B}}%
^{N}\left(  x\right)  $ for any $x\in\mathcal{D}$, we can define
\[
E_{\mathcal{A}}:\mathcal{A\rightarrow}%
{\textstyle\prod_{n\geq N}}
\mathcal{A}_{N}\subseteq%
{\textstyle\prod_{n\geq N}}
\mathcal{M}_{k_{N}}\left(  \mathbb{C}\right)
\]
by $E_{\mathcal{A}}\left(  a\right)  =\left(  E_{\mathcal{A}}^{N}\left(
a\right)  ,E_{\mathcal{A}}^{N+1}\left(  a\right)  ,\cdots\right)  $ and
\[
E_{\mathcal{B}}:\mathcal{B\rightarrow}%
{\textstyle\prod_{n\geq N}}
\mathcal{B}_{N}\subseteq%
{\textstyle\prod_{n\geq N}}
\mathcal{M}_{k_{N}}\left(  \mathbb{C}\right)
\]
by $E_{\mathcal{B}}\left(  a\right)  =\left(  E_{\mathcal{B}}^{N}\left(
a\right)  ,E_{\mathcal{B}}^{N+1}\left(  a\right)  ,\cdots\right)  $. It
follows that $E_{\mathcal{A}}$ and $E_{\mathcal{B}}$ are unital completely
positive maps, $E_{\mathcal{A}}|_{\mathcal{D}}=E_{\mathcal{B}}|_{\mathcal{D}}$
and $E_{\mathcal{A}}|_{\mathcal{D}}$ is a faithful unital *-homomorphism of
$\mathcal{D}$. From (\ref{e}), it is not hard to see that $E_{\mathcal{A}}$
and $E_{\mathcal{B}}$ are unital faithful representations from $\mathcal{A}$
and $\mathcal{B}$ into $%
{\textstyle\prod_{n\geq N}}
\mathcal{M}_{k_{N}}\left(  \mathbb{C}\right)  $ modulo the compacts
respectively. Then $\mathcal{A\ast}_{\mathcal{D}}\mathcal{B}$ is quasidiagonal
by Theorem \ref{8}.
\end{proof}

It is unknown whether the requirement that two AF algebras $\mathcal{A}_{i}$
have a pair of faithful traces $\tau_{i}$ on $\mathcal{A}_{i}$ which agree on
$\mathcal{D}$ is a necessary condition in proceeding theorem. But if the
embeddings from $\mathcal{D}$ into $\mathcal{A}_{i}$ $(i=1,2)$ are not unital,
then $\mathcal{A}_{1}\ast_{\mathcal{D}}\mathcal{A}_{2}$ could be QD, even
though there are no such faithful traces agree on $\mathcal{D}$. For example,
Let $\mathcal{D=}\mathbb{C}$. Suppose $\varphi_{1}:\mathbb{C}%
\mathcal{\rightarrow M}_{2}\left(  \mathbb{C}\right)  $ and $\varphi
_{2}:\mathbb{C}\mathcal{\rightarrow M}_{3}\left(  \mathbb{C}\right)  $ are
embeddings such that $\varphi_{1}\left(  1\right)  =\left(
\begin{array}
[c]{cc}%
1 & 0\\
0 & 0
\end{array}
\right)  $ and $\varphi_{2}\left(  1\right)  =\left(
\begin{array}
[c]{ccc}%
1 & 0 & 0\\
0 & 1 & 0\\
0 & 0 & 1
\end{array}
\right)  .$ Then $\mathcal{M}_{2}\left(  \mathbb{C}\right)  \underset
{\mathbb{C}}{\mathcal{\ast}}\mathcal{M}_{3}\left(  \mathbb{C}\right)
\cong\mathcal{M}_{2}\left(  \mathbb{C}\right)  \mathcal{\otimes M}_{3}\left(
\mathbb{C}\right)  $ is QD by Chapter 6 \cite{L}. When $\dim\mathcal{D}\geq2,$
we do not have a a satisfactory answer yet.

\begin{corollary}
\label{10}Suppose $\mathcal{A}$ is an AF algebra and $\mathcal{D}$ is a
finite-dimensional C*-subalgebra of $\mathcal{A}$. Then $\mathcal{A\ast
}_{\mathcal{D}}\mathcal{A}$ is quasidiagonal.
\end{corollary}

\begin{proof}
It is an easy consequence of Corollary \ref{9}.
\end{proof}

\begin{remark}
\label{11}Suppose that $\mathcal{A}$ is an AF algebra and $\mathcal{B=}%
\overline{\mathcal{\cup}_{i=1}^{\mathcal{1}}\mathcal{B}_{i}}\subseteq
\mathcal{A}$ is an AF subalgebra of $\mathcal{A}$. From Proposition 4.12 in
\cite{P}, we have that $\mathcal{A\ast}_{\mathcal{B}}\mathcal{A=}%
\lim\mathcal{A\ast}_{\mathcal{B}_{i}}\mathcal{A}$. But in general inductive
limits of QD C*-algebras may not be QD again. However, by Theorem 4
\cite{LS2}, we have that $\mathcal{A\ast}_{\mathcal{B}}\mathcal{A}$ is an MF algebra.
\end{remark}

\begin{center}
{\Large Acknowledgement}
\end{center}

The authors would like to thank the referee for his/her thorough review and
very useful comments.

\end{document}